\documentclass[utf8,reqno,final]{amsart}

\usepackage{inputenc}
\usepackage{latexsym}
\usepackage{mathrsfs}
\usepackage{amssymb}
\usepackage{euscript}
\usepackage[mathbf]{euler}
\usepackage{mathtools}
\usepackage{mathabx} 
\usepackage{paresse} 
\usepackage[all]{onlyamsmath}
\usepackage{ifthen}
\newcommand{\Ifempty}[3]{\ifthenelse{\equal{#1}{}}{#2}{#3}}

\usepackage[natbib=true,hyperref=true,backref=true,maxnames=2,backend=biber]{biblatex}
\makeatletter
\DeclareFieldFormat{postnote}{\blx@mkpageprefix{section}{#1}}
\makeatother

\usepackage{hyperref}
\usepackage{showlabels}
\showlabels{zlabel}
\usepackage{verbatim}
\usepackage[arrow,matrix]{xy}

\renewcommand{\AA}{\mathbb{A}}
\newcommand{\CC}{\mathbb{C}}

\newcommand{\Ee}{\mathcal{E}}

\newcommand{\Ii}{\mathcal{I}}

\newcommand{\Pp}{\mathcal{P}}

\newcommand{\Tt}{\mathcal{T}}

\newcommand{\Rt}{\mathrm{t}}

\newcommand{\gG}{\Gamma}

\newcommand{\w}{\omega}

\newcommand{\ul}[1]{\underline{#1}}
\newcommand{\ev}{\ul{ev}}
\newcommand{\comp}[1]{\ul{\circ #1}}
\newcommand{\ti}[1]{\widetilde{#1}}
\newcommand{\ra}[1][]{\xrightarrow{#1}}
\newcommand{\mt}{\mapsto}
\newcommand{\Ten}{\otimes}

\renewcommand{\1}{\mathbf{1}}
\renewcommand{\2}{\mathbf{2}}
\newcommand{\x}{\times}

\newcommand{\kk}{\Bbbk}

\newcommand{\Hom}{\ul{Hom}}

\newcommand{\B}{\ul{Bi}}

\newcommand{\Co}[1]{\widecheck{#1}}

\newcommand{\Def}[1]{\emph{#1}}

\DeclareMathAlphabet{\mathpzc}{OT1}{pzc}{m}{it}

\newcommand{\Pro}[2][]{\operatorname{Pro}_{#1}({#2})}
\newcommand{\Ind}[2][]{\operatorname{Ind}_{#1}({#2})}

\newcommand{\Lim}[2][]{\underset{#1}{\varprojlim}\,#2}
\newcommand{\CoLim}[2][]{\underset{#1}{\varinjlim}\,#2}

\DeclareMathOperator{\dcl}{dcl}
\DeclareMathOperator{\spec}{spec}

\newcommand{\cat}[1]{\mathcal{#1}} 
\newcommand{\ftr}[1]{\mathscr{#1}} 
\newcommand{\nat}[1]{§ #1} 
\newcommand{\obj}[1]{\mathbf{#1}} 
\newcommand{\dfs}[1]{\obj{#1}} 
\renewcommand{\th}[1]{\cat{#1}} 
\newcommand{\md}[1]{\ftr{#1}} 
\newcommand{\dcs}[1]{\md{#1}} 
\newcommand{\eem}[1]{\nat{#1}} 
\newcommand{\fM}{\ftr{M}}
\newcommand{\fI}{\ftr{I}}
\newcommand{\fF}{\ftr{F}}
\newcommand{\fG}{\ftr{G}}
\newcommand{\fA}{\ftr{A}}
\newcommand{\fH}{\ftr{H}}
\newcommand{\fN}{\ftr{N}}
\newcommand{\fS}{\ftr{S}}
\newcommand{\fR}{\ftr{R}}
\newcommand{\mM}{\md{M}}
\newcommand{\mN}{\md{N}}
\newcommand{\tT}{\th{T}}
\newcommand{\Cc}{\cat{C}}
\newcommand{\Dd}{\cat{D}}
\renewcommand{\Ee}{\cat{E}}
\renewcommand{\Ii}{\cat{I}}
\newcommand{\Sets}{\cat{S}et}
\newcommand{\Sch}{\cat{S}ch}
\newcommand{\oA}{\obj{A}}
\newcommand{\oB}{\obj{B}}
\newcommand{\oX}{\obj{X}}
\newcommand{\oY}{\obj{Y}}
\newcommand{\oZ}{\obj{Z}}
\newcommand{\oQ}{\obj{Q}}
\newcommand{\oG}{\obj{G}}
\newcommand{\oW}{\obj{W}}
\newcommand{\oU}{\obj{U}}

\newcommand{\oT}{\obj{T}}
\newcommand{\oC}{\obj{C}}
\newcommand{\oM}{\obj{M}}
\newcommand{\dX}{\dfs{X}}
\newcommand{\dZ}{\dfs{Z}}
\newcommand{\dW}{\dfs{W}}
\newcommand{\dE}{\dfs{E}}
\newcommand{\dF}{\dfs{F}}
\newcommand{\dY}{\dfs{Y}}
\newcommand{\dL}{\dfs{L}}
\newcommand{\pA}{\dcs{A}}
\newcommand{\pM}{\dcs{M}}
\newcommand{\ea}{\eem{a}}
\newcommand{\gt}{\nat{t}}
\newcommand{\Sc}[1]{\widehat{#1}}
\newcommand{\Op}[1]{#1^{op}}

\newcommand{\nref}[1]{\S\ref{#1}}

\newcommand{\newthm}[2]{\newtheorem{#1}[subsection]{#2}\newtheorem{#1s}[subsubsection]{#2}}
\newthm{theorem}{Theorem}
\newthm{claim}{Claim}
\newthm{cor}{Corollary}
\newthm{prop}{Proposition}
\newthm{lemma}{Lemma}
\theoremstyle{definition}
\newthm{dfn}{Definition}
\theoremstyle{remark}
\newthm{remark}{Remark}
\newthm{example}{Example}

\newcommand{\DefAlias}[2]{\expandafter\xdef\csname #1\endcsname{#2}}
\newcommand{\CiteAlias}[2]{\DefAlias{CITE#1}{#2}}
\renewcommand{\Cite}[2][]{\Ifempty{#1}{\citet{\csname CITE#2\endcsname}}{\citet[#1]{\csname CITE#2\endcsname}}}

\CiteAlias{Dgal}{MR1972449}
\CiteAlias{CWM}{MR1712872}
\CiteAlias{LNM900II}{LNM900II}
\CiteAlias{Topos}{MR1300636}
\CiteAlias{prodef}{prodef}
\CiteAlias{Rep2Dim}{MR2100671}
\CiteAlias{EtHom}{MR883959}
\CiteAlias{SGA4}{MR0354652}
\CiteAlias{SGA1}{MR2017446}
\CiteAlias{groupoids}{groupoid}
\CiteAlias{sacks}{MR0398817}
\CiteAlias{etale}{MR559531}
\CiteAlias{Deligne}{MR1106898}
\CiteAlias{acfgroups}{MR1827833}
\CiteAlias{defaut}{defaut}
\CiteAlias{tannakian}{tannakian}
\CiteAlias{Makkai}{MR0505486}
\CiteAlias{cassidy}{MR0360611}
\CiteAlias{cassidy2}{MR0409426}
\CiteAlias{freyd}{MR0166240}
\CiteAlias{ovchinnikov2}{AR0703422}
\CiteAlias{ovchinnikov1}{AR0702846}
\CiteAlias{Saavedra}{MR0338002}
\CiteAlias{DCF}{MR1773702}
\CiteAlias{DCFv}{MR2215060}
\CiteAlias{Waterhouse}{MR547117}
\CiteAlias{Kolchin}{MR0568864}
\CiteAlias{Zilberint}{MR606796}
\CiteAlias{PoizatGalois}{MR727805}
\CiteAlias{HrStable}{MR1081816}
\CiteAlias{simple}{MR2140036}
\CiteAlias{Pillay}{MR2094550}
\CiteAlias{Obch}{MR2480859}
\CiteAlias{Obchi}{MR2426137}
\CiteAlias{acfa}{MR1652269}
\CiteAlias{Marker}{MR1924282}
\CiteAlias{wikicat}{wikicat}
\title{A categorical approach to internality}
\author[M. Kamensky]{Moshe Kamensky}
\address{
  Department of Mathematics \\
  University of Notre-Dame \\
  Notre-Dame, IN 46556\\
  USA
}
\email{\url{mailto:mkamensky@nd.edu}}
\urladdr{\url{http://mkamensky.notlong.com}}
\subjclass[2010]{Primary 03C40,03G30,18C10}

\begin{document}
\begin{abstract}
  Model theoretic internality provides conditions under which the group of 
  automorphisms of a model over a reduct is itself a definable group. In this 
  paper we formulate a categorical analogue of the condition of internality, 
  and prove an analogous result on the categorical level. The model theoretic 
  statement is recovered by considering the category of definable sets.
\end{abstract}
\maketitle

\section{Introduction}
In model theory, one works with an abstract notion of a structure, or a 
model. Such a structure can be a set \(\pM\) (or a collection of sets), 
together with a distinguished collection of subsets of \(\pM\) and its 
powers, the definable subsets. An example could be an algebraically closed  
field, with the definable sets being the field operations, solutions to 
equations derived from them (polynomial equations).

There is a natural notion of maps of such structures, and so one may form the 
group \(Aut(\mM)\) of automorphisms (invertible such self-maps) of \(\mM\).  
This group and certain subgroups of it are an important tool in the study of  
definable sets and other model-theoretic objects. In general, the elements of 
this group have nothing to do with the elements of the structure \(\mM\) 
itself.  However, under special circumstances, certain such automorphism 
groups are \emph{themselves definable}: there is a canonical correspondence 
between the elements of the group and a definable subset in \(\mM\), and the 
(graph of the) group operation is also definable.  This condition is known as 
internality.  Internality also provides (under suitable assumptions) a 
definable Galois theory, that provides a correspondence between certain 
intermediate structures and definable subgroups.

Internality was first discovered in~\Cite{Zilberint}, where it was used to 
study strongly minimal sets. It can be used in an abstract setting to study 
the relation between definable sets that are ``close'' to each other.  
However, it was also discovered in~\Cite{PoizatGalois} that, when applied to 
a particular theory (differentially closed fields), the definable Galois 
group and the associated Galois theory specialise to a classical 
construction, the Galois group (and Galois theory) of linear differential 
equations. It has since been applied to provide and study Galois theory in 
various situations (e.g.~\Cite{Pillay}). A more detailed review of 
internality, along with some references appears in Section~\ref{mod}.

Whereas the original definition of internality appeared under rather specific 
model theoretic assumptions (strongly minimal structures), it was 
subsequently generalised and simplified. In~\Cite{Dgal} the construction 
appeared without any restrictions on the theories in question.

It is shown in~\Cite{tannakian}, that a certain categorical construction, the 
algebraic group associated with a Tannakian category, can also be viewed as 
an instance of this formalism. Since the theory of internality no longer 
depends on ``stability theoretic'' assumption, it was natural to ask whether 
it is possible to go in the other direction, and recover the theory of 
internality (or some variant) from a categorical statement.

The present paper is an attempt to answer this question. In the spirit 
of~\Cite{Makkai}, we replace a theory \(\tT\) by its category of definable 
sets, and interpretations and models by suitable functors. We do not use the 
precise characterisation provided by~\Cite{Makkai}, since our result holds in 
greater generality, but the motivation comes from there. On the other hand, 
though the category of definable sets is not a topos, when the theory 
eliminates imaginaries it is rather close being that, and we use some ideas 
from topos theory, and more generally from category theory, in the style 
of~\Cite{Topos}, that remain relevant in this case.

The main results of this paper appear in sections~\ref{int} 
and~\ref{progroup}. In Section~\ref{int}, we make some observations on the 
category of definable sets of a theory (some of these appear 
in~\Cite{Makkai}), and formulate categorical versions of stably embedded 
definable sets, and internality. We then prove a weak version of the 
existence of definable Galois groups in the categorical context 
(Corollary~\ref{int:cor}). In Section~\ref{progroup} we introduce additional 
assumptions on the data, that allow us to prove the full analogy to the model 
theoretic case. In Section~\ref{examples} we provide a sketch of two possible 
applications of these theorems.

The main point of this paper is to explicate some analogies between model 
theoretic notions and constructions with categorical counterparts. It is 
therefore my hope that this paper is readable by anyone with minimal 
knowledge in these areas. To this end, the paper includes two sections which 
are almost completely expository. Section~\ref{cat} reviews some basic 
notions from category theory, and discusses the categorical assumptions we 
would like to make. The only new result there is Proposition~\ref{cat:aut} 
(which is a direct application of ends, and is possibly known) and its 
corollary. Section~\ref{mod} contains a brief, and rather informal review of 
basic model theory, as well as internality and the associated Galois group.  
Hopefully, there are no new results there.

\section{Categorical background}\label{cat}
We review some notions from category theory that will be used in the 
description of internality. The classical reference for these is~\Cite{CWM}, 
though~\Cite{wikicat} is also sufficient. All categories in this paper are 
(essentially) small.

A contra-variant functor from a category \(\Cc\) to the category of sets will 
be called a \Def{presheaf} (on \(\Cc\)). The functor that assigns to an 
object \(\oX\) the presheaf \(\oY\mt{}Hom(\oZ,\oX)\) it represents is called 
the \Def{Yoneda embedding}, and the Yoneda lemma states that this embedding 
is fully faithful (induces a bijection on sets of morphisms). Hence we also 
denote by \(\oX\) the presheaf it represents, and view \(\Cc\) as a 
subcategory of its presheaf category whenever this is convenient. More 
generally, the Yoneda lemma states that \(\fF(\oX)=Hom(\oX,\fF)\) for any 
presheaf \(\fF\).

All categories we consider, with the exception of index categories, will have 
all finite (inverse) limits. If \(\Cc\) is such a category, we denote by 
\(\1\) the terminal object (empty product) of \(\Cc\), and by \(\gG\) (global 
sections) the functor \(\oX\mt{}Hom(\1,\oX)\).

\subsection{Pro- and Ind- objects}
Our categories will not, usually, be closed under infinite limits or 
colimits. However, certain such limits appear naturally in what follows.

We recall that a \Def{filtering category} is a category \(\Ii\) in which any 
two objects map into a common object, and for any two maps \(f,g:\oA\ra\oB\) 
there is some map \(h:\oB\ra\oC\) with \(h\circ{}f=h\circ{}g\). A 
\Def{filtering system} in a category \(\Cc\) is a functor from a filtering 
category to \(\Cc\). The category \(\Ind{\Cc}\) of \Def{ind-objects} of 
\(\Cc\) is a category that contains \(\Cc\) as a full sub-category, admits 
the colimits of all filtering systems, and satisfies
\begin{equation}
  Hom(\oX,\Ind[\Ii]{\oY_i})=\Lim[\Ii](Hom(\oX,\oY_i))
\end{equation}
where \(\oX\) and \(\oY_i\) are objects of \(\Cc\), and \(\Ind[\Ii]{\oY_i}\) 
is the colimit of the system in \(\Ind{\Cc}\) (note that if \(\Cc\) itself 
happens to admit the colimit \(\CoLim[\Ii]{\oY_i}\) of this system, this 
object will not be isomorphic to \(\Ind[\Ii]{\oY_i}\) in \(\Ind{\Cc}\)). A 
construction of this category appears in~\Cite{SGA4}.

The category \(\Pro{\Cc}\) of pro-objects is the dual notion: it is the 
opposite of the category of ind-objects of the opposite of \(\Cc\). Thus it 
consists of inverse limits \(\Pro{\oX_i}\) of co-filtering systems in 
\(\Cc\), with \(\Cc\) embedded as a full sub-category, and satisfying
\begin{equation}
  Hom(\Pro[\Ii]{\oX_i},\oY)=\Lim[\Ii](Hom(\oX_i,\oY))
\end{equation}
with similar conventions to the above.

Because of the above properties, the objects of \(\Cc\) will be called 
\Def{compact} when considered as objects of \(\Ind{\Cc}\) or \(\Pro{\Cc}\). A 
functor (or a presheaf) will be called pro-representable (ind-representable) 
if it is represented by a pro-~(ind-)~object. In other words, it is a 
filtered colimit of representable functors (presheaves).

By a \Def{pro-group} (resp., a pro-monoid, etc.) in \(\Cc\) we mean a 
pro-object of \(\Cc\) that can be represented by an inverse filtering system 
of group (resp. monoid) objects of \(\Cc\). Such a pro-group is a group 
object in \(\Pro{\Cc}\), but the converse is not true, in general: there may 
be group objects in \(\Pro{\Cc}\) which are not of this form.

\subsection{Closed and Ind-closed categories}
Let \(\Cc\) be a category with finite inverse limits. Given two objects 
\(\oX\) and \(\oY\) of \(\Cc\), we denote by \(\Hom(\oX,\oY)\) the presheaf 
\(\oZ\mt{}Hom(\oZ\x\oX,\oY)\). There is, therefore, a natural evaluation map 
\(\ev:\Hom(\oX,\oY)\x\oX\ra\oY\), where \(\oX\) and \(\oY\) are identified 
with the presheaf they represent. Applying the definition with \(\oZ=\1\) 
(and using Yoneda's lemma), we see that \(\gG(\Hom(\oX,\oY))=Hom(\oX,\oY)\).

The category \(\Cc\) is called (Cartesian) \Def{closed} (\Cite[IV.6]{CWM}) if 
all such \(\Hom\) presheaves are representable (in which case the 
representing object is denoted by the same symbol, and is called the 
\Def{internal \(Hom\)}).  In other words, the functor \(\oZ\mt\oZ\x\oX\) has 
a right adjoint \(\oY\mt\Hom(\oX,\oY)\). By the Yoneda lemma, the evaluation 
map is also in \(\Cc\) in this case.

\begin{dfn}\label{cat:indclosed}
  We will say that a category \(\Cc\) with finite inverse limits is 
  \Def{ind-closed} if \(\Hom(\oX,\oY)\) is ind-representable for all objects 
  \(\oX\) and \(\oY\).
\end{dfn}

Note that \(\Ind{\Cc}\) again has all finite inverse limits, but the 
condition does not imply that \(\Ind{\Cc}\) is closed.  Instead, we know that 
\(\Hom(\oX,\oY)\) exists (and is constructed in the obvious way) whenever 
\(\oX\) is compact.  Again, the evaluation map is in \(\Ind{\Cc}\).

It may sometimes be difficult to determine whether a category is closed.  
However, for ind-closed, we have the following formal criterion.

\begin{prop}
  Assume that \(\Cc\) has all finite limits and colimits. Then it is 
  ind-closed if and only if, for any object \(\oX\) of \(\Cc\), the functor 
  \(\oZ\mt\oZ\x\oX\) commutes with finite colimits.
\end{prop}
\begin{proof}
  The ``Grothendieck construction'' presents every presheaf as a colimit of 
  representable presheaves (\Cite[VII]{Topos}). The system produced by the 
  construction is filtering if the presheaf is left-exact (takes finite 
  colimits to the corresponding inverse limits). Hence in a category that has 
  finite colimits, the ind-representable presheaves are precisely the 
  left-exact ones (it is clear that any ind-representable presheaf is 
  left-exact).

  Since any representable presheaf is left-exact, \(\Hom(\oX,\oY)\) is 
  left-exact if \(-\x\oX\) preserves finite colimits.

  Conversely, if \(\Cc\) is ind-closed, then \(-\x\oX\) commutes with 
  colimits in \(\Ind{\Cc}\), since it has a right adjoint there. The 
  inclusion of \(\Cc\) in \(\Ind{\Cc}\) preserves finite colimits, so the 
  result follows.
\end{proof}

In our case, the requirement that finite colimits exist is too strong. It 
will be replaced by a more direct condition, elimination of imaginaries, 
ensuring that the system produced by the Grothendieck construction is 
filtering (see Proposition~\ref{int:ei}).

\subsection{Closed functors}\label{cat:cfunctors}
If \(\fF:\Cc\ra\Dd\) is any functor, it extends by composition to a functor 
\(\fF:\Ind{\Cc}\ra\Ind{\Dd}\). If \(\Cc\) is ind-closed, and \(\fF\) 
preserves finite inverse limits, we may apply \(\fF\) to the evaluation map, 
to obtain a map \(\fF(\Hom(\oX,\oY))\x\fF(\oX)\ra\fF(\oY)\). This induces a 
map \(\fF(\Hom(\oX,\oY))\ra\Hom(\fF(\oX),\fF(\oY))\), and we say that such a 
functor \(\fF\) is \Def{closed} if this map is an isomorphism.

\subsection{Endomorphisms of left exact functors}
We recall that a map (natural transformation) \(\gt\) from a functor 
\(\fF:\Cc\ra\Dd\) to another functor \(\fG:\Cc\ra\Dd\) consists of a 
collection morphisms \(\gt_\oX:\fF(\oX)\ra\fG(\oX)\), one for each object 
\(\oX\) of \(\Cc\), commuting with all maps that come from \(\Cc\).

Let \(\Cc\) and \(\Dd\) be two categories with finite inverse limits, and let 
\(\fF:\Cc\ra\Dd\) be a left exact functor. For an object \(\oT\) of \(\Dd\), 
we denote by \(End_\oT(\fF)\) the set of natural transformations from 
\(\oT\x\fF\) to \(\fF\) (where \((\oT\x\fF)(\oX)=\oT\x\fF(\oX)\)). Each such 
natural transformation induces a transformation of \(\oT\x\fF\) to itself, 
and therefore \(\oT\mt{}End_\oT(\fF)\) is a presheaf of monoids on \(\Dd\).  
We denote by \(Aut_\oT(\fF)\) the subset of \(End_\oT(\fF)\) consisting of 
transformations where the induced endomorphism above is invertible. Thus it 
is a presheaf of groups on \(\Dd\).

More generally, if \(\fI:\Cc_0\ra\Cc\) is another functor, we denote by 
\(End_\oT(\fF/\fI)\) and \(Aut_\oT(\fF/\fI)\) the subsets of \(End_\oT(\fF)\) 
and \(Aut_\oT(\fF/\fI)\) consisting of those endomorphisms whose 
``restriction'' to the composition \(\fF\circ\fI\) is the identity. These are 
again presheaves on \(\Dd\). The main observation in this section is that if 
\(\Dd\) is closed, these presheaves are represented by a pro-monoid and a 
pro-group, respectively, of \(\Dd\).

\begin{prop}\label{cat:aut}
  Let \(\fI:\Cc_0\ra\Cc\) and \(\fF:\Cc\ra\Dd\) be functors.  Assume that 
  \(\Dd\) admits finite inverse limits, and that \(\Hom(\fF(\oX),\fF(\oY))\) 
  exists for all objects \(\oX\) and \(\oY\) of \(\Cc\). Then the functors 
  \(\oT\mt{}End_\oT(\fF/\fI)\) and \(\oT\mt{}Aut_\oT(\fF/\fI)\) are 
  pro-representable, by a pro-monoid \(\oM\) and pro-group \(\oG\), 
  respectively.

  If \(\fH:\Dd\ra\Ee\) is any \emph{closed} functor (at least when restricted 
  to the essential image of \(\fF\)), then \(\fH(\oM)\) and \(\fH(\oG)\) 
  represent \(\oT\mt{}End_\oT(\fH\circ\fF/\fI)\) and 
  \(\oT\mt{}Aut_\oT(\fH\circ\fF/\fI)\), respectively.
\end{prop}

The proof is a direct consequence of the end construction, as explained 
below.

\begin{cor}\label{cat:main}
  Let \(\fI:\Cc_0\ra\Cc\) be a functor into an ind-closed category \(\Cc\).  
  Then there are pro-monoid \(\oM\) and pro-group \(\oG\) in \(\Ind{\Cc}\), 
  such that for any closed functor \(\fH:\Cc\ra\Ee\), 
  \(End(\fH/\fI)=\gG(\fH(\oM))\) and \(Aut(\fH/\fI)=\gG(\fH(\oG))\).
\end{cor}
\begin{proof}
  The embedding \(\fF\) of \(\Cc\) into \(\Dd=\Ind{\Cc}\) satisfies the 
  condition of the proposition.
\end{proof}

In the model theoretic application, \(\fH\) will be extension of constants 
(base-changes), so \(\gG\circ\fH\) is ``taking \(\mN\)-points'' for some 
model \(\mN\) (or simply a model, in the language of~\Cite{Makkai}; 
see~\nref{int:dcs}).

\subsection{Ends}
The following construction is presented in~\Cite[IX.5]{CWM}. Let 
\(\fS:\Op{\Cc}\x\Cc\ra\Dd\) be a bi-functor (so for any object \(\oX\) of 
\(\Cc\), \(\fS(\oX,-)\) and \(\fS(-,\oX)\) are a functor and a pre-sheaf, 
resp., with values in \(\Dd\)). The \Def{end} of \(\fS\) is an object \(e\) 
of \(\Dd\), together with a morphism \(p_{\oX}:e\ra\fS(\oX,\oX)\) for each 
object \(\oX\) of \(\Cc\), such that for any morphism \(f:\oX\ra\oY\) of 
\(\Cc\), the diagram
\begin{equation}
  \xymatrix{
    e\ar[r]^{p_{\oX}}\ar[d]_{p_{\oY}} & \fS(\oX,\oX)\ar[d]^{\fS(\oX,f)}\\
    \fS(\oY,\oY)\ar[r]_{\fS(f,\oY)} & \fS(\oX,\oY)
  }
\end{equation}
commutes, and which is universal with these properties.

\begin{example}
  Let \(\fF,\fG:\Cc\ra\Dd\) be two functors, and let 
  \(\fS:\Op{\Cc}\x\Cc\ra\Sets\) be the functor 
  \((\oX,\oY)\mt{}Hom(\fF(\oX),\fG(\oY))\). The set \(Hom(\fF,\fG)\) of 
  natural transformations between \(\fF\) and \(\fG\) maps to each 
  \(\fS(\oX,\oX)\) by restriction, and the commuting diagrams as above amount 
  to the statement that a collection of such restrictions amalgamates to a 
  natural transformation. Hence \(Hom(\fF,\fG)\) is the end of \(\fS\).
\end{example}

The end of a functor is in a natural way an inverse limit of diagrams as 
above. Hence it exists if the category \(\Dd\) admits all inverse limits. The 
proof of Proposition~\ref{cat:aut} is a repetition of the above example 
internally:

\begin{proof}[Proof of Proposition~\ref{cat:aut}]
Assume first that \(\Cc_0\) is the empty category. By assumption, we have a 
functor \(\fS:\Op{\Cc}\x\Cc\ra\Dd\), 
\(\fS(\oX,\oY)=\Hom(\fF(\oX),\fF(\oY))\). Since the category \(\Pro{\Dd}\) 
admits all filtered inverse limits, and \(\Dd\) (and therefore \(\Pro{\Dd}\)) 
admits all finite inverse limits, \(\Pro{\Dd}\) admits all inverse limits 
(\Cite[IX.1]{CWM}). Therefore, the end \(\oM\) of \(\fS\) exists. The fact 
that \(\oM\) represents \(\oT\mt{}End_\oT(\fF)\) follows directly from the 
definitions.

The fact that \(\oM\) can be given a system of monoid objects follows from 
the fact that \(End_\oT(\fF)\) is the filtered inverse limit of 
\(End_\oT(\fF_\alpha)\), where \(\fF_\alpha\) are restrictions of \(\fF\) to 
finite sub-categories \(\Cc_\alpha\) of \(\Cc\).

When \(\Cc_0\) is not empty, the pro-monoid in the theorem is the pullback of 
\(\oM\) constructed above along the product of the maps 
\(\1\ra\Hom(\fF(\fI(\oX)),\fF(\fI(\oX)))\) corresponding to the identity 
\(\fF(\fI(\oX))\ra\fF(\fI(\oX))\). Hence it is again of the same type.  
Likewise, the object \(\oG\) is a suitable pullback.

The last statement follows since in this case, by assumption, the functors 
\(\fH\circ\fS\) and \((\oX,\oY)\mt\Hom(\fH(\fF(\oX)),\fH(\fF(\oY)))\) are 
isomorphic, hence so are their ends.
\end{proof}

\section{Model theory and internality}\label{mod}
In this section we summarise some basic notions from model theory. Most of 
them are classical, and can be found, for example, in~\Cite{Marker}. The 
exception is the notion of internality, which is the central notion for us, 
and which is recalled here following~\Cite{groupoids}. The first three 
subsections are illustrated in Example~\ref{mod:acf}.

\subsection{Theories, models and definable sets}
Logic is concerned with studying mathematical theories as mathematical 
objects. Thus, the data of a \Def{theory} consists of a formal language, 
together with a collection of statements (axioms) written in this language.
A \Def{model} of the theory \(\tT\) is a structure consisting of an 
interpretation of the symbols in the language of \(\tT\), in which all the 
axioms of \(\tT\) hold.  For example, the theory of fields, and the theory 
\(ACF\) of algebraically closed fields can both be written in a language 
constructed of the symbols \((0,1,+,-,\cdot)\), and a model of \(ACF\) is a 
particular algebraically closed field (see Example~\ref{mod:acf}).

A \Def{formula} is written in the same formal language, but has \Def{free 
variables}, into which elements of the model can be plugged.  Any such 
formula \(\phi(x_1,\dots,x_n)\) (where \(x_1,\dots,x_n\) contain all free 
variables of \(\phi\)) thus determines a subset \(\phi(\mM)\) of \(\mM^n\), 
for any model \(\mM\), namely, the set of all tuples \(\bar{a}\) for which 
\(\phi(\bar{a})\) holds.  Two formulas \(\phi\) and \(\psi\) are 
\Def{equivalent} if \(\phi(\mM)=\psi(\mM)\) for all models \(\mM\). An 
equivalence class under this relation is called a \Def{definable set}. Thus, 
if \(\dX\) is a definable set, there is a well defined set \(\dX(\mM)\) for 
any model \(\mM\) (but the syntactic information is lost). If we have a fixed 
model \(\mM\) in mind, the set \(\dX(\mM)\) is also called definable.

If \(\mM\) is the model of some theory, the set of all statements (in the 
underlying language) that are true in \(\mM\) is a theory \(\tT(\mM)\), and 
\(\mM\) is a model of \(\tT(\mM)\).

We will assume all our theories to be multi-sorted, i.e., the variables of a 
formula can take values in any number of disjoint sets. In fact, if \(\dX\) 
and \(\dY\) are sorts, we view \(\dX\x\dY\) as a new sort. In particular, any 
definable set is a subset of some sort (for example, in \(ACF\) the sorts 
correspond to the affine spaces \(\AA^n\)).  By a statement such as 
\(a\in\mM\) we will mean that \(a\) is an element of one of the sorts.  
Likewise, a subset of \(\mM\) is collection of subsets of \(\dX(\mM)\), one 
for each sort \(\dX\).

\subsection{Parameters and definable closure}\label{mod:param}
Let \(\mM\) be a model of a theory \(\tT\), \(\dX\) and \(\dY\) definable 
sets. A function \(f:\dX(\mM)\ra\dY(\mM)\) is \Def{definable} if its graph is 
definable. An element \(b\in\mM\) is definable over another element \(a\) if 
\(f(a)=b\) for some definable function \(f\). The \Def{definable closure} 
\(\dcl(\pA)\) of a subset \(\pA\) of \(\mM\) is the set of elements definable 
over some tuple \(a\) of elements of \(\pA\). The set \(\pA\) is 
\Def{definably closed} if \(\pA=\dcl(\pA)\). We denote by \(\dX(\pA)\) the 
set \(\dX(\mM)\cap\dcl(\pA)\).

More generally, we may consider formulas \(\phi(x,a)\) with parameters in 
\(\pA\), obtained by plugging a tuple \(a\) from \(\pA\) in a regular formula 
\(\phi(x,y)\). These determine subsets in models containing \(\pA\), give 
rise to definable sets over \(\pA\). If \(\tT=\tT(\mM)\) is the theory of 
\(\mM\), it is possible to expand \(\tT\) by constant symbols and suitable 
axioms to obtain a theory \(\tT_\pA\), whose models are the models of \(\tT\) 
containing \(\pA\), and whose definable sets the \(\pA\)-definable sets of 
\(\tT\).

An \Def{elementary map} from a model \(\mM\) to another model \(\mN\) is a 
map \(\ea\) from \(\mM\) to \(\mN\) such that for each definable set \(\dX\), 
\(\dX(\mM)=\ea^{-1}(\dX(\mN))\). Such a map is necessarily an embedding, and 
when it is fixed, we identify elements of \(\mM\) with their image in 
\(\mN\). An \Def{automorphism} of \(\mM\) is an invertible elementary map.  
The automorphism \(\ea\) of \(\mM\) is over \(\pA\subseteq\mM\) if \(\ea\) 
fixes each element of \(\pA\).

An automorphism over \(\pA\) will also fix pointwise any element of 
\(\dcl(\pA)\). Hence \(\dcl(\pA)\) is contained in the set \(\mM^G\) of 
elements of \(\mM\) fixed pointwise by all elements of the group 
\(G=Aut(\mM/\pA)\) of automorphisms of \(\mM\) fixing \(\pA\) pointwise. This 
inclusion might be strict, in general, but any model admits an elementary 
embedding in a \Def{homogeneous} model, where this inclusion becomes an 
equality whenever the cardinality of \(\pA\) is smaller than that of the 
model.

\subsection{Imaginaries and interpretations}\label{mod:img}
A \Def{definable equivalence relation} on a definable set \(\dX\) is a 
definable subset of \(\dX\x\dX\) that determines an equivalence relation in 
any model. The theory \Def{eliminates imaginaries} if any equivalence 
relation has a quotient, i.e., any equivalence relation \(\dE(x,y)\) can be 
represented as \(f(x)=f(y)\) for some definable function \(f\) on \(\dX\).

A \Def{definition} of a theory \(\tT_1\) in another theory \(\tT_2\) is 
specified by assigning, to each definable set \(\dX\) of \(\tT_1\) a 
definable set \(\oX_2\) of \(\tT_2\), such that for any model \(\mM_2\) of 
\(\tT_2\), the sets \(\oX_2(\mM_2)\) comprise a model \(\mM_1\) of \(\tT_1\) 
(when interpreted as a \(\tT_1\) structure in the obvious way). The expansion 
\(\tT_\pA\) of \(\tT\) mentioned above is an example of a definition of 
\(\tT\) in \(\tT_\pA\).

For any theory \(\tT\) there is a theory \(\tT^{eq}\) and a definition of 
\(\tT\) in \(\tT^{eq}\), where \(\tT^{eq}\) eliminates imaginaries, and the 
definition is universal with this property. Any model of \(\tT\) can be 
expanded uniquely to a model of \(\tT^{eq}\). One can often assume that that 
a theory admits elimination of imaginaries via this process. Through the rest 
of this section, we assume that our theories eliminate imaginaries.

\begin{example}\label{mod:acf}
  We illustrate some of the notions discussed above for the theory \(ACF\) of 
  algebraically closed fields. All the fact in this example appear 
  in~\Cite{Marker}.
  
  As mentioned above, this theory is given in the language determined by 
  functions symbols \(+,-,\cdot\), and constant symbols \(0\) and \(1\). The 
  properties of an algebraically closed field can be expressed by (infinitely 
  many) statements in this language.  For instance, commutativity of addition 
  is written as \(\forall{}x\forall{}y(x+y=y+x)\).

  An example of a model is the field \(\CC\) of complex numbers. The full 
  theory \(\tT(\CC)\) of the complex numbers as a field turns out to be 
  \(ACF\) together with the (infinitely many) axioms stating that the 
  characteristic is \(0\).

  The sorts in this theory are all Cartesian powers of the field sort, hence 
  can be identified with \(\AA^n\), the affine \(n\)-space. A formula 
  consists of quantifiers (\(\forall\) and \(\exists\)) applied to boolean 
  combinations of polynomial equations. For example \(\exists{}y(y^2=x)\).  
  Such a formula determines a subset of some \(\AA^n\) (where \(n\) is the 
  number of free variables). Whereas the formulas may have an arbitrary 
  number of quantifiers, the subset it defines can also be defined by a 
  formula without quantifiers. In other words, the definable sets can all be 
  represented by quantifier-free formulas, i.e., they are the constructible 
  subsets.

  An automorphism is a bijection of the field with itself that preserves the 
  operations, \(0\) and \(1\). Hence it is the same as a field automorphism.  
  A set of parameters is a subset of some algebraically closed field.  Since 
  the field operations are definable, a definably closed set is a field, 
  which is also perfect, since extracting \(p\)-th roots (when \(p>0\) is the 
  characteristic) is again a definable function. By considering the 
  automorphism group of the model, we find that perfect fields are precisely 
  the definably closed sets: The definable closure of a set \(\pA\) is 
  contained in the fixed field of the group of automorphisms of the 
  algebraically closed field fixing \(\pA\). This fixed field is, by field 
  theory, the perfect field generated by \(\pA\).

  Let \(\dX\) be the definable set of pairs \((x,y)\). The definable subset 
  \(\dE((x,y),(z,w))\) of \(\dX\x\dX\) given by 
  \((x=z\land{}y=w)\lor(x=w\land{}y=z)\) is an equivalence relation, which is 
  definable not only in \(ACF\), but in any theory (since the formula 
  mentions only equality).  In the theory of equality (with no other symbols) 
  this relation does not have a quotient.  However, if \(\dX\) is a field, 
  the definable function \(f(x,y)=(x+y,xy)\) is a quotient. In fact, \(ACF\) 
  eliminates imaginaries.
\end{example}

\subsection{Internality}\label{mod:int}
Let \(\tT\) be an expansion of \(\tT_0\) (so we have an interpretation of 
\(\tT_0\) in \(\tT\)), and let \(\mM\) be a model of \(\tT\), with 
restriction \(\mM_0\) to \(\tT_0\). An automorphism of \(\mM\) over \(\mM_0\) 
consists certain bijections from \(\oX(\mM)\) to itself, for each definable 
set \(\oX\) of \(\tT\). These bijections will rarely come from 
\emph{definable} bijections in \(\tT\) (even with parameters). Internality, 
which is the main notion of interest for us, provides a situation where all 
such automorphisms are, in fact, definable (with parameters in \(\mM\)).

A definable set \(\oX\) in a theory \(\tT\) is \Def{stably embedded} if for 
any model \(\mM\) of \(\tT\), any \(\pM\)-definable subset of \(\oX^n\) (for 
any \(n\)) is definable over \(\oX(\pM)\). An interpretation of a theory 
\(\tT_0\) in \(\tT\) is a \Def{stable embedding} if any definable set of 
\(\tT_0\) is stably embedded when viewed as a definable set in \(\tT\).

Given a stable embedding of \(\tT_0\) in \(\tT\), let \(\mM\) be a model of 
\(\tT\) of cardinality bigger than the language, and let \(\mM_0\) be the 
induced model of \(\tT_0\). Then \(\tT\) is an \Def{internal cover} of 
\(\tT_0\) if for any such model there is a small subset \(\pA\) of \(\mM\) 
(smaller than the cardinality of \(\mM\)), such that 
\(\dcl(\mM_0\cup\pA)=\mM\). A set \(\pA\) of this form can (and will) always 
be taken to be definably closed, and will be called a set of \Def{internality 
parameters}. We will assume for simplicity that such an \(\pA\) can be found 
with \(\pA_0=\pM_0\cap\pA\) contained in the definable closure of \(0\) in 
\(\tT_0\) (the theory works also without this assumption, but then one has to 
deal with definable groupoids, rather than groups. Cf.~\Cite{groupoids}).

\begin{example}
  Let \(\tT\) be the theory generated by two sorts, \(\dL\) and \(\dX\), and 
  stating that \(\dL\) is an algebraically closed field, and \(\dX\) is a 
  vector space over \(\dL\), of a fixed dimension \(n\) (in the natural 
  language for such a structure). Then the theory \(\tT_0\) of algebraically 
  closed fields is stably embedded in \(\tT\) via \(\dL\), and \(\tT\) is an 
  internal cover of \(\tT_0\), since in any model \(\mM\) of \(\tT\), the 
  definable closure of any basis of \(\dX(\mM)\) over \(\dL(\mM)\) is a set 
  of internality parameters.
\end{example}

The internal covers are interesting because of the following theorem, which 
was first proved by Zil'ber (\Cite{Zilberint}) in the context of strongly 
minimal structures. The version here was proved in~\Cite[appendix~B]{Dgal}, 
but was reformulated in (approximately) our language in~\Cite{groupoids}.

\begin{theorem}\label{mod:main}
  Let \(\tT\) be  an internal cover of \(\tT_0\). There is a pro-group 
  \(\oG\) in \(\tT\), together with a definable action 
  \(m_Q:\oG\x\oQ\ra\oQ\) of \(\oG\) on every definable set 
  \(\oQ\) of \(\tT\), such that for any model \(\mM\) of 
  \(\tT\), \(\oG(\mM)\) is identified with \(Aut(\mM/\mM_0)\) 
  through this action (with \(\mM_0\) the restriction of \(\mM\) to 
  \(\tT_0\)).
\end{theorem}

Another part of this theorem, which will not be discussed in this paper, 
provides a Galois correspondence. As mentioned in the introduction, 
\Cite{PoizatGalois} developed this Galois correspondence to recover the 
Galois theory of a linear differential equation. Other applications appear 
in~\Cite{Pillay}, \Cite{Dgal} and~\Cite{defaut}.  In~\Cite{tannakian}, this 
theorem is the main tool in a model theoretic proof of the Tannakian 
formalism.

\section{Categorical internality}\label{int}
Let \(\tT\) be a first order theory. The collection of definable sets, with 
definable maps between them, forms a category, \(\Dd_\tT\). Our aim in this 
section is to reformulate some of the notions and results of 
Section~\ref{mod}, and especially the notion of internality, in terms of the 
category \(\Dd_\tT\). We in fact work with more general categories, and 
obtain a weaker result than Theorem~\ref{mod:main}. In 
Section~\ref{progroup}, we introduce additional assumptions on the category, 
and deduce the full strength of the theorem.

\subsection{The category of definable sets}\label{int:defsets}
The structure of the category \(\Dd_\tT\) was described in~\Cite{Makkai}.  We 
will not need the full characterisation, since our result holds under more 
general and simpler assumptions. We mention, however, that a category of the 
form \(\Dd_\tT\) admits all finite inverse limits, and finite 
co-products\footnote{We allow disjoint unions of definable sets as a basic 
logical operation. This is automatic if there are at least two constant 
symbols, or if \(\tT\) eliminates imaginaries, and is convenient otherwise.  
It implies that \(\Dd_\tT\) has finite co-products.}.  A \Def{logical 
functor} between such categories preserves these limits (and has additional 
properties). When \(\tT\) is eliminates imaginaries, \(\Dd_\tT\) is a special 
kind of a \Def{pre-topos}, in the sense of~\Cite{SGA4}. In fact, the 
construction of \(\tT^{eq}\) appeared first in~\Cite{Makkai}, as a universal 
embedding of a logical category in a pre-topos.

The categories \(\Pro{\Dd_\tT}\) and \(\Ind{\Dd_\tT}\) were discussed 
in~\Cite{prodef}. Essentially the results are that one may evaluate \(\mM\) 
points of a pro- or ind- definable set by evaluating them for each member in 
a system, and taking the corresponding limit. Definable maps between such 
sets are then precisely those that come from maps in the corresponding 
category, and if the model is sufficiently saturated, a definable map is an 
isomorphism if it is bijective on \(\mM\)-points.

We denote by \(\1\) and \(\2\) the definable sets with \(1\) and \(2\) 
(named) elements. These are the final object in \(\Dd_\tT\) and its coproduct 
with itself, respectively.

\begin{prop}\label{int:ei}
  The theory \(\tT\) eliminates imaginaries if and only if \(\Dd_\tT\) is 
  ind-closed (in the sense of Definition~\ref{cat:indclosed}).
\end{prop}
\begin{proof}
  Assume \(\Dd_\tT\) is ind-closed, and  let \(\dE\subseteq\dX\x\dX\) be an 
  equivalence relation. The corresponding characteristic function 
  \(\chi_\dE:\dX\x\dX\ra\2\) corresponds, by assumption, to some 
  \(t:\dX\ra\Pp(\dX):=\Hom(\dX,\2)\). If \(\Pp(\dX)\) is represented by some 
  system \((\dY_a,f_{a,b})\), the map \(t\) is represented by some 
  \(t_a:\dX\ra\dY_a\), and so the relation \(\dE(x,y)\) is equivalent to the 
  intersection over all \(b\) of the conditions 
  \(f_{a,b}(t_a(x))=f_{a,b}(t_a(y))\), hence, by compactness, by one such 
  condition. The map \(f_{a,b}\circ{}t_a\) onto its image is then the 
  quotient.

  Conversely, assume that \(\tT\) eliminates imaginaries, and let \(\dX\) and 
  \(\dY\) be two definable sets. Let \(\Ii\) be the category whose objects 
  are definable maps \(\dZ\x\dX\ra\dY\), and whose morphisms are definable 
  maps of the \(\dZ\) coordinate that make the obvious diagram commute.  For 
  each such object, we view the elements of \(\dZ\) as maps from \(\dX\) to 
  \(\dY\). Let \(\bar{\dZ}\) be the quotient of \(\dZ\) by the definable 
  equivalence relation saying that \(z_1,z_2\in\dZ\) define the same map.  
  The quotient map determines a map in \(\Ii\) from \(\dZ\x\dX\ra\dY\) to 
  \(\bar{\dZ}\x\dX\ra\dY\), which co-equalises any two maps of \(\Ii\) into 
  its domain.  Since \(\Ii\) also inherits the co-products from \(\Dd_\tT\), 
  it has all finite colimits, and is therefore filtering category. The 
  ``forgetful'' functor that assigns to an object \(\dZ\x\dX\ra\dY\) of 
  \(\Ii\) the definable set \(\dZ\) in \(\tT\) is therefore a filtering 
  system in \(\Dd_\tT\). It is clear that this system represents  
  \(\Hom(\dX,\dY)\).
\end{proof}

\begin{remark}
  The system used in the second part of the above proof is the one used in 
  the Grothendieck construction, to show that any presheaf is a colimit of 
  sheaves. The existence of quotients produces the co-equalisers in this 
  system, despite the fact that the original category \(\Dd_\tT\) might  
  not have admitted them.
\end{remark}

A model \(\mM\) of \(\tT\) determines a functor \(\dX\mt\dX(\mM)\) from 
\(\Dd_\tT\) to the category of sets, which preserves all inverse limits,  
co-products, and images.  Conversely, any such functor is (isomorphic to) a 
model.  An elementary map from \(\mM\) to \(\mN\) as models is the same as a 
map of the corresponding functors. A subset \(\pA\) of \(\mM\) is definably 
closed if and only if it defines a left-exact sub-functor of \(\mM\) that 
preserves co-products.

Similarly, a definition of the theory \(\tT_1\) in another theory \(\tT_2\) 
is the same as a logical functor \(\fI\) (in the same sense as above) from 
\(\Dd_{\tT_1}\) to \(\Dd_{\tT_2}\). Given a model \(\mM\) of \(\tT_2\) 
(viewed as a functor), the corresponding model of \(\tT_1\) is 
\(\mM\circ\fI\) (See also~\Cite[7]{Makkai}). We will from now on identify 
models with the corresponding functors (so that \(\dX(\mM)=\mM(\dX)\) for a 
definable set \(\dX\) and a model \(\mM\)), and likewise for definably closed 
sets and definitions.

The following result is essentially the same as~\Cite[Theorem~7.1.8]{Makkai}.

\begin{prop}
  Let \(\fI:\Dd_{\tT_1}\ra\Dd_{\tT_2}\) be a definition, where \(\tT_1\) 
  eliminates imaginaries, and let \(\fI^*\) be the induced functor from 
  models of \(\tT_2\) to models of \(\tT_1\) (obtained by composition). Then 
  \(\fI\) is an equivalence of categories if and only if \(\fI^*\) is.
\end{prop}
\begin{proof}
  Assume that \(\fI^*\) is an equivalence. We will show that \(\fH\) is 
  essentially surjective. Let \(\dX\) be a definable set of \(\Tt_2\), and 
  assume first that \(\dX\subseteq\fI(\dY)\) for some definable set \(\dY\) 
  of \(\tT_1\). Then for any two models \(\mM\) and \(\mN\) of \(\Tt_2\) such 
  that \(\mM\circ\fI=\mN\circ\fI\), we have \(\mM(\dX)=\mN(\dX)\) (by 
  equivalence).  Hence by Beth's definability theorem, \(\dX=\fI(\dX_0)\) for 
  some \(\dX_0\).

  We next claim that for any model \(\fM_2\) of \(\tT_2\), any element of 
  \(\fM_2\) is definable over the restriction \(\fM_1=\fM_2\circ\fI\).  
  Indeed, \(\fM_2\) embeds into a sufficiently saturated model \(\fN_2\), and 
  if there is an automorphism of \(\fN_2\) that fixes pointwise \(\fM_1\) but 
  not \(\fM_2\), we get two different embeddings of \(\fM_2\) in \(\fN_2\), 
  which agree on \(\fM_1\), contradicting the equivalence.

  It follows, by compactness, that any definable set \(\dY\) of \(\tT_2\) is 
  the image of a definable map from a definable set \(\fI(\dX)\) that comes 
  from \(\tT_1\). By the previous point, the kernel of this map also comes 
  from \(\tT_1\), hence, since \(\tT_1\) eliminates imaginaries, \(\dY\) 
  itself comes from \(\tT_1\) as well.

  This shows that \(\fI\) is essentially surjective. It is full since any map 
  is a definable set, and it is faithful since for equality of definable sets 
  is first order. Consequently, \(\fI\) is an equivalence.  The other 
  direction is trivial.
\end{proof}

As indicated in~\nref{mod:param}, one often considers sets definable with 
parameters in a set \(\pA\). These can be viewed as definable sets in an 
expansion \(\tT_\pA\) of \(\tT\). The expansion corresponds to a functor 
\(\fI_\pA:\Dd_\tT\ra\Dd_{\tT_\pA}\).

\begin{prop}\label{int:expclosed}
Assume \(\Tt\) eliminates imaginaries, and let \(\mM\) be a model of \(\tT\).  
Then for any definably closed \(\pA\subseteq\pM\), the functor \(\fI_\pA\) is 
closed (in the sense of~\nref{cat:cfunctors}).
\end{prop}
\begin{proof}
  A definable map \(f\) in \(\tT_\pA\) from \(\dZ\) to 
  \(\fI_\pA(\Hom(\dX,\dY))\) is given by formulas 
  \(\phi(w,z)\subseteq\dW\x\dZ\) and 
  \(\dF\subseteq\dW\x\dZ\x\fI_A(\Hom(\dX,\dY))\), such that 
  \(\dZ=\phi(a,-)\), \(f=\dF(a,-,-)\), \(a\in\dW(\dcs{A})\) and 
  \(\dF,\dW,\phi\) are in \(\tT\). There is a \(\tT\) definable subset 
  \(\dW_0\subseteq\dW\) such that \(\dF(w,-,-)\) is a function from 
  \(\phi(w,-)\) to \(\Hom(\dX,\dY)\) for all \(w\in\dW_0\). We may assume 
  that \(\dW=\dW_0\), so that \(\dF\) gives a \(\tT\) definable function from 
  \(\phi\) to \(\Hom(\dX,\dY)\). Thus it corresponds to a map 
  \(\ti{F}:\phi\x\dX\ra\dY\). The restriction \(\ti{f}:\dZ\x\dX\ra\dY\) of 
  this map to \(\phi(a,-)\x\dX\) corresponds to the original map \(f\).

  Reversing this argument gives, for each map \(\ti{f}:\dZ\x\dX\ra\dY\) in 
  \(\tT_\pA\) a corresponding map \(f:\dZ\ra\fI_\pA(\Hom(\dX,\dY))\). This 
  shows that \(\fI_\pA(\Hom(\dX,\dY))\) represents the correct functor in 
  \(\tT_\pA\), and therefore coincides with 
  \(\Hom(\fI_\pA(\dX),\fI_A(\dY))\).

  (More abstractly, \(\tT_\pA\) can be viewed as a full sub-category of 
  \(\Pro{\tT}\), and \(\fI_\pA\) as the inclusion of \(\tT\) in 
  \(\Pro{\tT}\).  This inclusion is always closed.)
\end{proof}

\subsection{}\label{int:dcs}
Since the definable closure of the empty set can be identified with the set 
of one element definable subsets, we get that for each definable set, 
\(\dX(0)=\gG(\dX)=Hom(\1,\dX)\). Likewise, for any subset \(\pA\) of a model 
\(\mM\) we have
\begin{equation}
\fA(\dX)=\dX(\pA)=\gG_\pA(\fI_\pA(\dX))
\end{equation}
(where \(\gG_\pA\) is global sections in \(\Dd_{\tT_\pA}\)). In other words, 
\(\fA=\gG_\pA\circ\fI_\pA\).  We deduce that \(\Hom(\dX,\dY)\) gives the 
expected result when evaluated on parameters.

\begin{cor}\label{int:hom}
  Assume \(\tT\) eliminates imaginaries. For any definably closed set \(\pA\) 
  and any definable sets \(\dX\) and \(\dY\), the evaluation map identifies 
  \(\Hom(\dX,\dY)(\pA)\) with the set of \(\pA\)-definable maps from \(\dX\) 
  to \(\dY\). In particular, \(\Pp(\dX)(\pA)\) is the set of 
  \(\pA\)-definable subsets of \(\dX\).
\end{cor}
\begin{proof}
  According to Proposition~\ref{int:expclosed} and the above remark,
  \begin{multline}
    \Hom(\dX,\dY)(\pA)=\gG_\pA(\fI_\pA(\Hom(\dX,\dY)))=\\
    =Hom_\pA(1,\Hom(\dX,\dY))=Hom_\pA(\dX,\dY)
  \end{multline}
\end{proof}

Finally, we can characterize the stable embeddings.

\begin{prop}\label{int:stem}
  Let \(\fI:\Dd_{\tT_1}\ra\Dd_{\tT_2}\) be an interpretation, where both 
  theories eliminate imaginaries. Then the interpretation is stably embedded 
  if and only if \(\fI\) is closed.
\end{prop}
\begin{proof}
  Assume that the interpretation is stably embedded, let \(\dX\), 
  \(\dY\) be definable sets in \(\tT_1\), and let \(\mM\) be a model of 
  \(\tT_2\). The map \(\Rt:\fI(\Hom(\dX,\dY))\ra\Hom(\fI(\dX),\fI(\dY))\) 
  determines a map
  \begin{equation*}
    t:\fI(\Hom(\dX,\dY))(\mM)\ra\Hom(\fI(\dX),\fI(\dY))(\mM)
  \end{equation*}
  By Corollary~\ref{int:hom}, the co-domain is the set of \(\mM\)-definable 
  maps from \(\fI(\dX)\) to \(\fI(\dY)\) in \(\tT_2\), while the domain is 
  the set of definable maps from \(\dX\) to \(\dY\) in \(\tT_1\) with 
  parameters from the model \(\mM\circ\fI\) (and \(t\) assigns to each map 
  the same map viewed as a map between \(\fI(\dX)\) and \(\fI(\dY)\)). Since 
  \(\tT_1\) asserts that all maps in \(\Hom(\dX,\dY)\) are distinct, the same 
  has to hold in \(\tT_2\).  Therefore, \(t\) is injective. Since \(\tT_1\) 
  is stably embedded, any \(\mM\)-definable map between \(\fI(\dX)\) and 
  \(\fI(\dY)\) is definable with parameters in \(\mM\circ\fI\), hence \(t\) 
  is surjective.  Since \(t\) is a bijection in every model, \(\Rt\) is an 
  isomorphism.

  Assume conversely that \(\fI\) is closed, let \(\dX\) be a definable set in 
  \(\tT_1\), and let \(\mM\) be a model of \(\tT_2\). Any \(\mM\)-definable 
  subset \(\dZ\subseteq\fI(\dX)\) (in \(\tT_2\)), corresponds to an 
  \(\mM\)-point of \(\Pp(\fI(\dX))\) (by Corollary~\ref{int:hom}), hence to 
  an \(\mM\circ\fI\) point of \(\Pp(\dX)\), which can be used to define 
  \(\dZ\).
\end{proof}

\subsection{Internality}\label{int:aut}
We now return to discuss categories \(\Cc\) which are ind-closed and admit  
finite co-products. Thus, \(\Cc=\Dd_\tT\) for a theory \(\tT\) with 
elimination of imaginaries is an example. Our aim is to formulate the notion 
of internality in~\nref{mod:int} in this setting. From 
Proposition~\ref{int:stem}, we already have the notion of a stably embedded 
interpretation: This is simply a closed functor between two such ind-closed 
categories.

Let \(\tT\) be an internal cover of \(\tT_0\) (both eliminating imaginaries), 
and let \(\oX\) be a definable set in \(\tT\). By assumption (and 
compactness), there is a parameter \(a\), and an \(a\)-definable surjective 
map \(f_a\) from some sort in \(\tT_0\) onto \(\oX\). Since \(\tT_0\) 
eliminates imaginaries, we may assume \(f_a\) to be bijective. The domain 
\(\oX_0\) of \(f_a\) is a definable set in \(\tT_0\), hence (since \(\tT_0\) 
is stably embedded) it can be defined with a (canonical) parameter from a 
model of \(\tT_0\), and therefore (because of the assumption in the end 
of~\nref{mod:int}) with no parameters at all.

Thus, after fixing suitable parameters of internality, we may associate to 
each definable set \(\oX\) in \(\tT\), a definable set \(\oX_0\) in 
\(\tT_0\). Since this process commutes with products and inclusions and 
complements, this is an interpretation of \(\tT\) in \(\tT_0\). Furthermore, 
this is a stably embedded interpretation: any subset of \(\oX_0\) definable 
with parameters can be translated via the same function \(f_a\) to a subset 
definable with parameters of \(\oX\). This motivates the following 
definition.

\begin{dfn}\label{int:def}
Let \(\Cc_0\) be an ind-closed category. An \Def{internal cover} of \(\Cc_0\) 
consists of an ind-closed category \(\Cc\), together with closed functors 
\(\fI:\Cc_0\ra\Cc\) and \(\fF:\Cc\ra\Cc_0\), and an isomorphism of 
\(\fF\circ\fI\) with the identity functor (we will ignore this isomorphism in 
the notation and assume that \(\fF\circ\fI\) is the identity functor).
\end{dfn}

\begin{remark}
The assumption that \(\fF\circ\fI\) is the identity is equivalent (under the 
other conditions) to \(\fF\) being ``internally left adjoint'' to \(\fI\), in 
the sense that the canonical map 
\(\Hom(\fF(\oX),\oY_0)\ra\fF(\Hom(\oX,\fI(\oY_0)))\) is an isomorphism (this 
can be seen by applying the isomorphism with \(\oX=\1\)). The ``unit of 
adjunction'' obtained by setting \(\oY_0=\fF(\oX)\) (which does not formally 
exist, in general) is, in the model theoretic setting, the collection of maps 
\(f_a\) described above (which ``do not exist'' because they require 
parameters).
\end{remark}

We now wish to prove an analogue of Theorem~\ref{mod:main}. According 
to~\nref{int:dcs}, a ``model'' of \(\Cc_0\) is given by a composition of 
functors \(\gG_\Dd\circ\fH\), where \(\fH\) is a closed functor into an 
ind-closed category \(\Dd\), and \(\gG_\Dd\) is the functor \(Hom(\1,-)\) in 
\(\Dd\). A model of \(\Cc\) is thus given by \(\gG_\Dd\circ\fH\circ\fF\). We 
are interested in automorphisms of this functor that fix \(\fI\). We also 
know from Corollary~\ref{cat:main} that the group \(Aut(\fH\circ\fF/\fI)\) of 
``definable'' automorphisms comes from a pro-group, but \(\gG_\Dd\) is far 
from being closed. Regardless, we have the following statement.

\begin{theorem}\label{int:main}
  Let \(\fI:\Cc_0\ra\Cc\), \(\fF:\Cc\ra\Cc_0\) be an internal cover (so 
  \(\Cc_0\) and \(\Cc\) are ind-closed).  Then for any left-exact 
  \(\fH:\Cc_0\ra\Dd\), the map of groups
  \begin{equation*}
    Aut(\fH\circ\fF/\fI)\ra Aut(\gG_\Dd\circ\fH\circ\fF/\fI)
  \end{equation*}
  obtained by applying \(\gG_\Dd\) is an isomorphism.
\end{theorem}
We note that \(\fH\) need not be closed in this theorem.
\begin{proof}
  Let \(\gt\) be an automorphism of \(\fR=\gG_\Dd\circ\fH\circ\fF\) fixing 
  \(\fR\circ\fI\), let \(\oQ\) be an object of \(\Cc\), and set 
  \(\oX=\Hom(\oQ,\fI(\fF(\oQ)))\). Thus there is an evaluation map 
  \(\ev:\oX\x\oQ\ra\fI(\fF(\oQ))\). By the remark above, we have 
  \(\fF(\oX)=\Hom(\fF(\oQ),\fF(\oQ))\), and hence an evaluation map 
  \(\fF(\oX)\x\fF(\oQ)\ra\fF(\oQ)\). Since \(\fR\) is left exact, we thus get 
  a map
  \begin{equation*}
    \fR(\ev):\fR(\oX)\x\fR(\oQ)\ra\fR(\oQ)
  \end{equation*}
  In particular, each element of \(\fR(\oX)\) determines a map of 
  \(\fR(\oQ)\) to itself. We claim that \(\gt_\oQ\) is represented by an 
  element of \(\fR(\oX)\), \(g:\1\ra\fH(\fF(\oX))\). This will show that 
  \(\gt\) comes from an automorphism of \(\fH\circ\fF\), namely the 
  composition
  \begin{equation*}
    \fH(\fF(\oQ))=\1\x\fH(\fF(\oQ))\ra[g\x{}Id]\fH(\fF(\oX))\x\fH(\fF(\oQ))\ra
    \fH(\fF(\oQ))
  \end{equation*}

  We first note that there is a map \(\1\ra\fF(\oX)\) in \(\Cc_0\) 
  corresponding to the identity on \(\fF(\oQ)\). Applying \(\gG\circ\fH\), we 
  deduce that the identity map is represented by an element \(1\) in 
  \(\fR(\oX)\). Since \(\oX\), \(\oQ\), \(\fI(\fF(\oQ))\) and the evaluation 
  map are from \(\Cc\), the evaluation map \(\fR(\ev)\) is preserved by 
  \(\gt\). Also, \(\gt_{\fI(\fF(\oQ))}\) is the identity. Hence, for all 
  \(q\in\fR(\oQ)\), we get 
  \(\fR(\ev)(\gt_\oX(1),\gt_\oQ(q))=\fR(\ev)(1,q)=q\). It follows that 
  \(\gt_\oX(1)\), as an endomorphism of \(\fR(\oQ)\) is invertible, with 
  inverse \(\gt_\oQ\). This complete the proof of the surjectivity. The 
  injectivity is similar, using the composition map instead of the 
  evaluation.
\end{proof}

\begin{cor}\label{int:cor}
  Let \(\fI:\Cc_0\ra\Cc\), \(\fF:\Cc\ra\Cc_0\) be an internal cover. There is 
  a pro-group \(\oG\) in \(\Ind{\Cc}\), and an action of \(\oG\) on each 
  object of \(\Cc\), such that for any closed functor \(\fH:\Cc_0\ra\Dd\) 
  into an ind-closed category, 
  \(\gG_\Dd(\fH(\fF(\oG)))=Aut(\gG_\Dd\circ\fH\circ\fF/\fI)\).
\end{cor}
\begin{proof}
  By Corollary~\ref{cat:main} we have such a \(\oG\) with the property that 
  \(\gG_\Dd(\fH(\fF(\oG)))=Aut(\fH\circ\fF/\fI)\). By Theorem~\ref{int:main}, 
  the last group is equal to \(Aut(\gG_\Dd\circ\fH\circ\fF/\fI)\).
\end{proof}

As explained above, an internal cover of theories gives rise to an internal 
cover in our sense. Applying the corollary with \(\fH\) an expansion by 
constants in a model recovers Theorem~\ref{mod:main}, but with a pro-group in 
ind-definable, rather than definable sets.

\section{Recovering a pro-group}\label{progroup}
The model theoretic statement in Theorem~\ref{mod:main} produces a 
pro-definable group, i.e., a pro-group in the category \(\Cc\) of definable 
sets, whereas the more general statement of Corollary~\ref{int:cor} provides 
only a pro-group in \(\Ind{\Cc}\).  This difference seems unavoidable in 
general, but in this section we would like to formulate additional conditions 
that ensure the stronger statement (these  conditions are easily seen to hold 
in the first order setting).

The idea, which appears in the model theoretic proof, is that the objects 
\(\oX=\Hom(\oQ,\oQ)\) which participate in the construction of \(\oG\) are 
unnecessarily big.  If \(\oY\) is an object in a system that represents 
\(\oX\), it will not, in general, be closed under composition. However, if 
\(f_\oQ\) is an ``element'' of \(\oY\) that is part of a natural 
automorphism, then composition with it remains within \(\oY\), since the map 
from \(\oY\) to \(\oX\) is preserved by the natural automorphism. Hence, to 
construct the natural automorphisms, it is possible to replace \(\oX\) by a 
compact sub-object.

For this to work, we need to know that the internal \(Hom\)s can be 
constructed from compact sub-objects. This is captured by the following 
definitions.

\begin{dfn}
  A map \(f:\oX\ra\oY\) in \(\Ind{\Cc}\) is called \Def{proper} if for any 
  \(g:\oZ\ra\oY\) with \(\oZ\) compact, the pullback \(\oX\x_\oY\oZ\) is also 
  compact.

  An ind-object \(\oY\) is called \Def{semi-strict} if any map from a compact  
  object to \(\oY\) is proper.
\end{dfn}

For example, if \(\oY\) is compact, then \(f:\oX\ra\oY\) is proper if and 
only if it is in \(\Cc\).

Recall that an ind-object \(\oY\) is \Def{strict} if any map \(\oX\ra\oY\), 
with \(\oX\) compact, factors as \(\oX\ra\oZ\ra[i]\oY\), with \(\oZ\) compact  
and \(i\) a monomorphism (equivalently, \(\oY\) can be represented by a 
system of monomorphisms).

\begin{lemma}
  Let \(\oY\) be an ind-object.
  \begin{enumerate}
    \item If \(\oY\) is strict, then it is semi-strict\label{stsst}
    \item \(\oY\) is semi-strict if and only if for any two ind-objects 
      \(\oX\) and \(\oZ\) over \(\oY\), the map \(\oX\x_\oY\oZ\ra\oX\x\oZ\) 
      is proper.
  \end{enumerate}
\end{lemma}
The converse of~\eqref{stsst} seems to require additional assumptions (for 
example, quotients by equivalence relations).

\begin{proof}
  \begin{enumerate}
    \item Since \(\oY\) is strict, we may assume that both arrows from 
      compact objects \(\oX\) and \(\oZ\) factor through the same 
      monomorphism \(\oW\ra\oY\), with \(\oW\) compact. Hence 
      \(\oX\x_\oY\oZ=\oX\x_\oW\oZ\) is compact as well.
    \item This immediately reduces to the case when \(\oX\) and \(\oZ\) are 
      compact, in which case we need to show that \(\oX\x_\oY\oZ\) is 
      compact, which is precisely the definition of \(\oY\) being 
      semi-strict.\qedhere
  \end{enumerate}
\end{proof}

For two objects \(\oX\) and \(\oY\), let \(\B(\oX,\oY)\) be the ind-object 
consisting of invertible maps from \(\oX\) to \(\oY\), and let 
\(\B(\oX)=\B(\oX,\oX)\) (this is a group ind-object).

\begin{theorem}\label{thm:progroup}
  In the situation of Corollary~\ref{int:cor}, assume in addition:
  \begin{enumerate}
    \item For any two objects \(\oX\) and \(\oY\) of \(\Cc\), the ind-object 
      \(\Hom(\oX,\oY)\) is semi-strict.
    \item In \(\Cc_0\), each \(\B(\oX,\oY)\) is strict.\label{t:strict}
    \item If \(\oX\) is an object of \(\Cc\) such that \(\fF(\oX)\) is 
      compact, then \(\oX\) is compact.\label{t:compact}
    \item For any monomorphism \(k:\oX\ra\oY\) and object \(\oZ\), all in 
      \(\Cc_0\), the induced map \(\comp{k}:\Hom(\oZ,\oX)\ra\Hom(\oZ,\oY)\) 
      is proper.\label{t:proper}
  \end{enumerate}
  Then the group \(\oG\) of the conclusion of Corollary~\ref{int:cor} is a 
  pro-group in \(\Cc\).
\end{theorem}

Note that assumption~\eqref{t:compact} holds if \(\fF\) reflects 
isomorphisms, in the sense that \(i\) is an isomorphism if \(\fF(i)\) is.

\begin{proof}
  The proof is analogous to the model theoretic proof. The idea is that by 
  internality, for any object \(\oQ\) of \(\Cc\) there is an object \(\oC\) 
  of \(\Cc_0\) with \(\B(\oQ,\oC)\) ``non-empty''. The group \(\oG\) acts 
  freely on this ind-object, and therefore, by assumption, on a compact 
  sub-object \(\oX\). \(\oG\) is then obtained by a sequence of pullbacks as 
  in the proof of Proposition~\ref{cat:aut}, except that now the functor 
  whose end we are taking has values in \(\Cc\) itself for objects on the 
  diagonal. The pullbacks are again compact by the properness assumption.

  More precisely, in the construction of the end, the pro-group \(\oG\) is 
  constructed as an inverse system of pullbacks of the form 
  \(\oG_3=\oG_1\x_{\Hom(\oX_1,\oX_2)}\oG_2\), where \(\oG_i\) is a subgroup 
  of \(\B(\oX_i)\), and the product is with respect composition maps with a 
  given morphism \(f:\oX_1\ra\oX_2\) in \(\Cc\) (\(\oG_3\) is then viewed as 
  a subgroup of \(\B(\oX_1\x\oX_2)\)).  Since we are assuming that 
  \(\Hom(\oX_1,\oX_2)\) is semi-strict, \(\oG_3\) will be compact if 
  \(\oG_1\) and \(\oG_2\) are.  Hence, it is enough to prove that each 
  \(\B(\oQ)\) contains a compact subgroup \(\oG_\oQ\), which is in the 
  system.

  Let \(\oC=\fI(\fF(\oQ))\) (so that \(\fF(\oC)=\fF(\oQ)\)), and let 
  \(\ti{t}:\oX\ra\B(\oQ,\oC)\) be a map from a compact object in \(\Cc\), 
  such that the map \(\1\ra\Hom(\fF(\oC),\fF(\oC))\) corresponding to the 
  identity map on \(\fF(\oC)\) factors through \(\fF(\ti{t})\). Thus, we have 
  a corresponding evaluation map \(t:\oX\x\oQ\ra\oC\), and a map 
  \(i:\1\ra\fF(\oX)\), such that \(\fF(t)\circ(i\x{}id_{\fF(\oQ)})\) is the 
  identity. We also have a ``composition with the inverse'' map 
  \(m:\oX\x\oX\ra\B(\oC)\), given on points \((x,y)\in\oX\x\oX\) by 
  \(m(x,y)=x\circ{}y^{-1}\). We let \(\oG_\oQ\) be the ``subgroup'' of 
  \(\B(\oX)\) preserving \(m\); More precisely, \(\B(\oX)\) embeds in 
  \(\B(\oX\x\oX)\) by acting diagonally, and composition with \(m\) gives in 
  total a map \(\B(\oX)\ra\Hom(\oX\x\oX,\oC)\). \(\oG_\oQ\) is the pullback 
  of this map along the map \(\check{m}:\1\ra\Hom(\oX\x\oX,\oC)\) 
  corresponding to \(m\).
  \begin{equation}
    \xymatrix{
      \oG_\oQ\ar[r]\ar[d]&\B(\oX\x\oX)\ar[d]^-{\comp{m}}\\
      \1\ar[r]_-{\check{m}} & \Hom(\oX\x\oX, \oC)
    }
  \end{equation}

  The ind-group \(\oG_\oQ\) is one of the groups that appear in the system 
  defining the total group of automorphisms. Thus, it is enough to prove that 
  it is compact. For that, it is sufficient, by the assumption, to prove that 
  \(\fF(\oG_\oQ)\) is compact. Hence we may assume that \(\oX\) is a 
  subobject \(k:\oX\ra\B(\fF(\oC))\) containing the identity \(i:\1\ra\oX\), 
  and \(\oG_\oQ\) is a group acting on \(\oX\). The identity map \(i\) and 
  the action \(\mu\) of \(\oG_\oQ\) on \(\oX\) induce a map 
  \(j=\mu\circ(id\x{}i):\oG_\oQ\ra\oX\).  The pullback diagram defining 
  \(\oG_\oQ\) implies that we have a pullback diagram
  \begin{equation}
    \xymatrix{
      \oG_\oQ\ar[r]^{\check{\mu}}\ar[d]_{j} &\Hom(\oX,\oX)\ar[d]^-{\comp{k}} \\
      \oX\ar[r]_-{\check{\circ}} & \Hom(\oX, \B(\fF(\oC)))
    }
  \end{equation}

  Since \(k\) is a monomorphism, the vertical map \(\comp{k}\), induced by 
  composition with \(k\), is by assumption proper. Therefore \(\oG_\oQ\), 
  being the pullback over a compact object of a proper map, is itself 
  compact.
\end{proof}

\section{Examples}\label{examples}
As explained above, Corollary~\ref{int:cor} and~\ref{thm:progroup} are 
categorical analogues of the classical model theoretic  
Theorem~\ref{mod:main}.  I now briefly mention two additional examples where 
the more general results may be applied.

\subsection{Partial automorphisms}
Instead of considering automorphisms of the full model \(\fM\), as in 
Theorem~\ref{mod:main}, one may be interested in automorphisms of a definable 
set \(\oQ\) in \(\tT\) that preserve only part of the structure on \(\oQ\).  
This question was studied in~\Cite{defaut}, where one is ultimately 
interested in automorphisms of a linear difference equation that preserve 
only the quantifier free formulas.

The categorical version may be applied in this case by taking \(\Cc\) to be 
the sub-category consisting of the definable sets one wishes to preserve. The 
condition that this category eliminates imaginaries corresponds to the 
requirement that the internality datum that witnesses the fact that 
\(\dcl(A\cup\fM_0)=\fM\), should be preserved.

\subsection{Tannakian categories}
As mentioned above, the model theoretic theorem can be applied to prove a 
version the Tannakian reconstruction theorem. Below I outline the 
``composition'' of the model theoretic argument with the categorical 
formulation of internality described above. This is not a complete example, 
but a ``work plan'', since some of the assumptions may be difficult to 
verify. The idea is only to show how the language used here is translated in 
this case.

Let \(\kk\) be a field of characteristic \(0\). Recall (\Cite[7]{Deligne}) 
that if \(\ti{\Cc}\) is a (symmetric, rigid, \(\kk\)-linear) tensor category, 
one may construct the category \(\Sch_{\ti{\Cc}}\) of \(\ti{\Cc}\)-schemes as 
the opposite category of the category of (commutative, unital) algebras in 
\(\Ind{\ti{\Cc}}\) (such an algebra  consists of an ind-object \(A\) of 
\(\ti{\Cc}\), together with morphisms \(m:A\Ten{}A\ra{}A\) and 
\(u:\1\ra{}A\), satisfying obvious axioms). This category admits finite 
inverse limits, since a pullback corresponds to a tensor product of algebras 
(over a base), as well as finite co-products.

Furthermore, there is a faithful functor \(\oX\mt\Sc{\oX}\) from \(\ti{\Cc}\) 
to \(\Sch_{\ti{\Cc}}\), sending \(\oX\) to the symmetric algebra on the dual 
\(\Co{\oX}\). Let \(\Cc\) be the full sub-category of \(\Sch_{\ti{\Cc}}\) 
containing all object of the form \(\Sc{\oX}\x_{\Sc{\oY}}\Sc{\oZ}\), with 
\(\oX\), \(\oY\) and \(\oZ\) objects of \(\ti{\Cc}\), as well as their finite 
co-products (these are the ``finitely presented'' \(\ti{\Cc}\)-schemes). It 
again admits all finite inverse limits and co-products. It is possible to 
show, like in usual algebraic geometry, that \(\Hom(\oX,\oY)\) is a strict 
ind-object for all objects \(\oX\) and \(\oY\) of \(\Cc\).

In particular, we take \(\ti{\Cc_0}\) to be the category of 
finite-dimensional vector spaces over \(\kk\), so that \(\Cc_0\) is the 
category of finitely presented affine schemes over \(\kk\). An exact tensor 
functor between tensor categories induces a left-exact functor between the 
corresponding categories of schemes, which restricts to a functor between the 
finitely presented sub-categories.  Furthermore, the construction of the 
\(\Hom\) ind-schemes shows that the resulting functor is a stable embedding.  
In particular, the (essentially unique) tensor functor from \(\ti{\Cc_0}\) to 
\(\ti{\Cc}\)  induces a functor \(\fI:\Cc_0\ra\Cc\).

If \(\ti{\Cc}\) is neutral Tannakian, i.e., we have an exact tensor functor 
\(\w\) to \(\ti{\Cc_0}\), this functor again induces a stable embedding 
\(\fF:\Cc\ra\Cc_0\), with \(\fF\circ\fI\) isomorphic to the identity. We thus 
get an internal cover, in the sense of Definition~\ref{int:def}.

If \(\oT\) is any (affine) scheme over \(\kk\), setting \(\fH_\oT\) to be 
base change by \(\oT\) (i.e., \(\fH_\oT(\oU)=\oT\x\oU\) as a functor from 
\(\kk\)-schemes to \(\oT\)-schemes), one gets from Corollary~\ref{int:cor} 
that any automorphism of the functor \(\oX\mt\fF(\oX)(\oT)\) (as a functor 
from \(\Cc\) to sets) comes from a unique compatible collection of 
automorphisms of the \(\oT\)-schemes \(\oT\x\fF(\oX)\); and furthermore, all 
such automorphisms correspond to the \(\oT\)-points of (with the help of 
Theorem~\ref{thm:progroup}, whose assumptions should be verified) a 
pro-group-scheme over \(\kk\), acting on all \(\fF(\oX)\). In fact, this 
group comes from a pro-group in \(\Cc\) (which is called the fundamental 
group of \(\Cc\) in~\Cite{Deligne}).

Note that if \(\oX\) is the \(\ti{\Cc}\)-scheme associated to an object 
\(\ti{\oX}\) of \(\ti{\Cc}\), then \(\fF(\oX)\) is \(\w(\ti{\oX})\) with its 
\(\kk\)-scheme structure. Hence, if \(\oT=\spec(A)\) is an affine scheme over 
\(\kk\), then \(\fF(\oX)(\oT)=\fF(\oX)(A)\) is identified (at least as a set) 
with \(A\Ten_\kk\w(\ti{\oX})\). Furthermore, the linear space and tensor 
structures are described by suitable maps between schemes, all of which come 
from \(\Cc\). Therefore, an automorphism of the functor 
\(\oX\mt\fF(\oX)(\oT)\) as a functor to sets is the same as a tensor 
automorphism of the functor \(\oX\mt{}A\Ten_\kk\w(\oX)\).

\printbibliography[maxnames=10]

\end{document}